\documentclass[11pt]{article}
\usepackage{amsmath,amsthm,amsfonts,amssymb,amscd, amsxtra}
\usepackage{url}
\usepackage[margin=2.5cm,nohead]{geometry}
\usepackage{cite}
\usepackage{color}
\newtheorem{theorem}{Theorem}
\newtheorem{lemma}{Lemma}

\newtheorem{proposition}{Proposition}
\newtheorem{remark}{Remark}

\newtheorem{example}{Example}

\def\argmin{\operatorname{argmin}}
\def\min{\operatorname{Minimize}}
\def\const{\operatorname{subject~to~}}

\DeclareMathOperator{\Tol}{Tol}

\DeclareMathOperator{\diag}{diag}
\DeclareMathOperator{\sgn}{sgn}

\newcommand{\lng}{\langle}
\newcommand{\rng}{\rangle}

\newcommand{\R}{\mathbb R}

\newcommand{\NN}{\mathbb{N}}

\begin{document}

\title{A semi-smooth Newton method for  a special  piecewise linear system with  application to  positively  constrained   convex quadratic   programming}

\author{ J. G. Barrios \thanks{IME/UFG, Avenida Esperança, s/n Campus Samambaia, 
Goi\^ania, GO, 74690-900, Brazil (e-mail:{\tt
numeroj@gmail.com}).  The author was supported in part by CAPES.}  
\and
J.Y. Bello Cruz\thanks{IME/UFG, IME/UFG, Avenida Esperança, s/n Campus Samambaia, 
Goi\^ania, GO, 74690-900, Brazil (e-mail:{\tt
yunier@ufg.br}).  The author was supported in part by
FAPEG, CNPq Grants 303492/2013-9,  474160/2013-0 and PRONEX--Optimization(FAPERJ/CNPq).}
\and
 O. P. Ferreira\thanks{IME/UFG, IME/UFG, Avenida Esperança, s/n Campus Samambaia, 
Goi\^ania, GO, 74690-900, Brazil (e-mail:{\tt
orizon@ufg.br}).  The author was supported in part by
FAPEG, CNPq Grants 4471815/2012-8,  305158/2014-7 and PRONEX--Optimization(FAPERJ/CNPq).}  
\and
S. Z. N\'emeth \thanks{School of Mathematics, The University of Birmingham, The Watson Building, 
Edgbaston, Birmingham B15 2TT, United Kingdom
(e-mail:{\tt nemeths@for.mat.bham.ac.uk}). The author was supported in part by 
the Hungarian Research Grant OTKA 60480.} }

\maketitle

\begin{abstract}
\noindent In this paper  a special piecewise linear system is studied. It is shown that,  under a mild assumption,  the  semi-smooth Newton method applied to this system  is  well defined and  the method  generates a  sequence  that  converges linearly to a solution. Besides, we also show that  the generated  sequence is   bounded,  for any  starting point,  and a formula for any accumulation point of this sequence is presented. As an application,  we study the convex quadratic programming problem  under   positive constraints. The numerical  results suggest that the semi-smooth Newton method achieves accurate solutions to large scale  problems in few iterations. \\



\end{abstract}

\section{Introduction}
 In this paper we consider  the following   special  piecewise linear system:
\begin{equation}\label{eq:pwls}
		x^+ +Tx=b, 
\end{equation} 
where,  denoting by  $\R^{n \times n}$ the set of $n \times n$ matrices with real entries and  $\R^n\equiv \R^{n \times 1}$    the $n$-dimensional Euclidean space,  the data consists of   $ b$ a   vector in $\R^n$,   $T $   a  nonsingular  matrix in $\R^{n \times n}$,   the  variable $x$ is  a vector in $\R^n$ and  $x^+$    is the  vector in $\R^n$ with  $i$-th component equal to $(x_ i)^+=\max\{x_ i,0\}$.  In \cite{BrugnanoCasulli2007} was proposed a semi-smooth Newton's method for solving \eqref{eq:pwls}.  Under suitable assumption was  showed the finite convergence to a solution of  \eqref{eq:pwls}.   Some works dealing with \eqref{eq:pwls} and its generalizations    include \cite{BrugnanoCasulli2007, BrugnanoCasulli2009, BrugnanoSestini2011, ChenRavi2010, GriewankBerntRadonsStreubel2015, SunLeiLiu2015}.  It is worth  mentioning that a similar equation has been studied  in \cite{Mangasarian2009}.

The  purpose  of  the present   paper  is   to discuss  the  semi-smooth Newton's method   introduced in \cite{BrugnanoCasulli2007}, to solve \eqref{eq:pwls},   under  new assumptions. As an application, we use the obtained results  to  study  the remarkable instance of  \eqref{eq:pwls}, 
\begin{equation}\label{equation}
		\left[ Q-{\rm I}\right]x^+ +x=- {\tilde b}, 
\end{equation}
where  the data consists of  $Q $  a positive definite real  matrix  of size $n\times n $ and   ${\tilde b}\in \R^n $.  Moreover, we  present  some computational experiments designed to investigate its practical viability.  It is worth pointing out that  the semi-smooth Newton's method  for solving \eqref{equation}  was studied in \cite{FerreiraNemeth2014} and   some computational tests were presented in \cite{BarriosFerreiraNemeth2015}. The results obtained in this  paper  improve the ones of \cite{FerreiraNemeth2014}.   As we will show, the system \eqref{equation} arises from the optimality condition of the  convex quadratic programming problem  under a  positive constraint, 
\begin{align}  \label{eq:napsc2}
 & \min   ~ \frac{1}{2}x^\top Qx+x^\top {\tilde b} +c\\  \notag
 & \const  x \in \R^n_+, 
\end{align}
where   $c $ is a real number and  $ \R^n_+ $ is the nonnegative orthant. Note that,  without loss of generality,   we can assume $Q$  symmetric in \eqref{eq:napsc2}   because the objective function of \eqref{eq:napsc2} is equal to $\frac{1}{2}x^\top \tilde Qx+x^\top {\tilde b} +c$, where $\tilde Q=\frac{1}{2}(Q+Q^T)$ is a symmetric matrix. Positively  constrained convex quadratic  programming  is equivalent to  the problem of   projecting the point   onto a  simplicial cone.  The interest in the subject of projection arises in several situations,   having a wide range of applications in pure and applied mathematics such as  Convex Analysis 
(see \emph{e.g.}, \cite{HiriartLemarecal1}),   Optimization (see  \emph{e.g.}, \cite{BuschkeBorwein96,censor07,censor01, scolnik08}), Numerical Linear Algebra (see \emph{e.g.}, \cite{Stewart77}), Statistics  (see \emph{e.g.}, \cite{BerkMarcus96,Dykstra83,Xiaomi1998}), Computer Graphics (see \emph{e.g.}, \cite{Fol90}) and  Ordered 
Vector Spaces (see \emph{e.g.}, \cite{AbbasNemeth2012,IsacNem86,IsacNem92, Nemeth20091,Nemeth2010-2}).    The projection onto a  general simplicial cone is difficult and computationally expensive, this problem has been studied \emph{e.g.}, in \cite{AlSultanMurty1992,EkartNemethNemeth2009,Frick1997,MurtyFathi1982,NemethNemeth2009}.  It is a special convex quadratic
program and its  KKT optimality conditions consists in  a linear complementarity problem (LCP) associated with it,  see \emph{e.g.},
\cite{Murty1988,MurtyFathi1982}.  Therefore, the problem of projecting onto  simplicial cones can be solved by  active set
methods \cite{Bazaraa2006,LiuFathi2011,LiuFathi2012,Murty1988} or any algorithms for solving LCPs,  see \emph{e.g.},  \cite{Bazaraa2006,Murty1988} and
special methods based on its geometry,   see \emph{e.g.}, \cite{MurtyFathi1982,Murty1988}. Other fashionable ways to solve this problem are based on
the classical von Neumann algorithm (see \emph{e.g.},  Dykstra algorithm \cite{DeutschHundal1994,Dykstra83,Shusheng2000}). Nevertheless, these 
methods are also quite expensive (see the numerical results in \cite{Morillas2005} and the remark preceding Section 6.3 in 
\cite{MingGuo-LiangHong-BinKaiWang2007}).

Following the ideas  of \cite{Mangasarian2009},    we show that the approach using semi-smooth Newton's method,   for solving \eqref{eq:napsc2},   has potential advantages over existing methods. The main advantage appears to be the global,  linear convergence and  to achieve accurate solutions  of large scale problems in few iterations. Our numerical results suggest,   for a given class of problem, that the number of required iterations is almost unchanged.  The numerical results also indicate a remarkable  robustness   with respect  to the starting point. 
  
  The organization of the paper is as follows.  In Section~\ref{sec:int.1},  some notations and preliminaries  used in the paper are presented.  In Section~\ref{sec: defscls} we study the convergence properties of  the semi-smooth Newton's method   for solving \eqref{eq:pwls}.  In Section~\ref{sec: defbp}  the results of Section~\ref{sec: defscls} are applied  to find a solution of \eqref{eq:napsc2}. In Section~\ref{sec:ctest} we present  some computational tests. Some final remarks are made in Section~\ref{sec:conclusions}.
\subsection{Notations and preliminaries} \label{sec:int.1}
In this subsection we present  the notations and  some auxiliary results used throughout the paper.
Let $\R^n$ be   the $n$-dimensional Euclidean space with  the canonical inner  product
$\lng\cdot,\cdot\rng$  and induced norm  $\|\cdot\|$.    The  $i$-th component of a vector $x
\in\R^n$ is denoted by $x_i$ .  We use the partial ordering for vectors,  defined by $ x\leq y$  meaning  $ x_i\leq y_i$, for all $i=1, \ldots, n$.
  For $x\in \R^n$, $\sgn(x)$ will denote a vector with components equal to $1$, $0$ or $-1$ depending on whether the corresponding component of the
vector $x$ is positive, zero or negative.    If $a\in\R$ and $x\in\R^n$, then denote $a^+:=\max\{a,0\}$, $a^-:=\max\{-a,0\}$ and $x^+$ and  $x^-$  the vectors with  $i$-th component equal to $(x_ i)^+$ and  $(x_ i)^-$, respectively. From the definitions of $x^+$ and  $x^-$ we have $x=x^+- x^-$,  $ \langle  x^+,   x^-\rangle=0$ and  $x^+,  x^- \in \R^n_+$. 
\begin{lemma}  \label{l:error}
Let  $x, y\in \R^n$. Then  $ \left\| y^+ - x^+-\emph{diag}(\emph{sgn}(x^+))(y-x)\right\|\leq \|y-x\|$.
\end{lemma} 
\begin{proof}
For  each  $i\in \{1,\ldots, n\}$, we have two possibilities:
\begin{itemize}
\item[(a)]  $x_i< 0$. In this case,  $\mbox{sgn}(x_i^+)=0$. Thus, $| y_i^+ - x_i^+-\mbox{sgn}(x_i^+)(y_i-x_i)|= |y_i^+|\leq |y_i-x_i|$.  
\item[(b)] $x_i\geq 0$. In this case,  $\mbox{sgn}(x_i^+)=1$. Hence,  $| y_i^+ - x_i^+-\mbox{sgn}(x_i^+)(y_i-x_i)|= |y_i^+-y_i|\leq |y_i-x_i|$. 
\end{itemize}
Combining (a) and (b) we have $( y_i^+ - x_i^+-\mbox{sgn}(x_i^+)(y_i-x_i))^2\leq (y_i-x_i)^2$,  for all  $i=1,\ldots, n$, which implies the desired inequality.
\end{proof}
  The matrix ${\rm I} \in \R^{n \times n}$ denotes the  identity matrix.  If $x\in \R^n$ then $\diag (x)\in \R^{n \times n}$ will denote a  diagonal matrix with $(i,i)$-th entry equal
to $x_i$, $i=1,\dots,n$. Denote $\|M\|:=\max \{\|Mx\|~:~  x\in \R^{n}, ~\|x\|=1\}$ for any  $M \in \R^{n\times n}$.  The next useful result was proved in  2.1.1,  page 32  of \cite{Ortega1990}.
\begin{lemma}\label{lem:ban}
Let $E \in \R^{n\times n}$. If  $\|E\|<1$,  then $E-{\rm I}$
is invertible and  $ \|(E-{\rm I})^{-1}\|\leq 1/\left(1-
\|E\|\right). $
\end{lemma}
 We end this section with the contraction mapping principle  (see 8.2.2,  page 153  of \cite{Ortega1990}).
\begin{theorem}  [contraction mapping principle] \label{fixedpoint}
 Let  $\phi : \R^{n} \to   \R^{n}$.  Suppose that there exists $\lambda \in  [0,1)$ such that  $\|\phi(y)-\phi(x)\| \le \lambda \|y-x\|$, for all $ x, y \in  \R^{n}$. Then there exists a unique $\bar x\in  \R^{n}$  such that $\phi(\bar x) = \bar x$.
\end{theorem}
\section{A semi-smooth Newton method for  a   piecewise linear systems} \label{sec: defscls}

In this section we  present  and  analyze  the semi-smooth Newton's method for solving \eqref{eq:pwls}.  We begin with an existence result of solution to the equation \eqref{eq:pwls}.
\begin{proposition} \label{pr:uniqqpw} 
Let $ \lambda \in \mathbb{R}$.  If $\left\|T^{-1}\right\| \leq \lambda<1 $  then   \eqref{eq:pwls}   has unique solution  for any $b\in \R^n$.
\end{proposition}
  \begin{proof}  The equation  \eqref{eq:pwls}   has a solution  if only if 
 $
 \phi (x)=-T^{-1}x^+ + T^{-1}b 
 $
 has a fixed point. It follows from definition of  $ \phi$  that
 \[
 \phi(y)- \phi(x)= -T^{-1} (y^+ -x^+ ), \qquad \qquad   ~x, y \in \mathbb{R}^n.
\]
Since   $\left\|T^{-1}\right\|<\lambda<1 $,  the last equality implies that $ \| \phi(y)- \phi(x)\|\leq \lambda \|y-x\|,  $ for all   $ x, y \in \mathbb{R}^n$. Hence  $ \phi$ is a contraction. Therefore  applying Theorem~\ref{fixedpoint} we conclude that $\phi $ has  precisely a unique fixed point and consequently   \eqref{eq:pwls}     has  a unique solution.
\end{proof}
Then  the assumption   $\left\|T^{-1}\right\| <1 $   in Proposition~\ref{pr:uniqqpw}  is sufficient to the  uniqueness of solution of
\eqref{eq:pwls}.   The next example shows that it is not possible  to increase the upper  bound of $\left\|T^{-1}\right\|$  and still ensure the uniqueness of solution in \eqref{eq:pwls}. 
\begin{example}
Consider the  function $F: \mathbb{R}^2 \to \mathbb{R}^2$ defined by  $F(x)=x^+ + Tx- b$, where
$$
T=  \begin{bmatrix}
-1 & 0 \\
0 & 1
\end{bmatrix},  \qquad 
b= \begin{bmatrix}
0\\
2
\end{bmatrix}.
$$
Note that  $\|T^{-1}\|=1$ and there holds $F(x^*)=F(x^{**})=0$, where $ x^*= [1,  1]^T$ and $x^{**}=[0, 1 ]^T$.
\end{example}
\noindent  The {\it semi-smooth Newton method} introduced in   \cite{LiSun93}   for  finding the zero of the function 
\begin{equation} \label{eq:fucpw}
F(x):=x^+ + Tx- b,   \qquad \qquad   ~x \in \mathbb{R}^n, 
\end{equation}
with starting point   $x^{0}\in \mathbb{R}^n$, it  is formally  defined by 
\begin{equation} \label{eq:nmqcpw}
F(x^{k})+ V^k\left(x^{k+1}-x^{k}\right)=0,  \qquad   V^k \in \partial F(x^{k}), \qquad k=0,1,\ldots,
\end{equation}
where  $ V^k$ is any subgradient in  $ \partial F(x^{k})$  the  Clarke generalized Jacobian of $F$ at $x^{k}$  (see   Definition~ 2.6.1 on page 70 of  \cite{Clarke1990}).  Letting 
\begin{equation} \label{eq:mppw}
P(x):=\mbox{diag}(\mbox{sgn}(x^+)),  \qquad x\in  \R^n, 
\end{equation} 
it easy to see that 
\[
P(x) +T\in  \partial F(x), \qquad x\in  \R^n. 
\]
Since  $ P(x)x=x^+ $ for all $x\in \R^n$,       taking  $V^k=P(x^{k}) +T$,    equation  \eqref{eq:nmqcpw}   becomes   
\begin{equation}\label{eq:newtonc2pw}
	\left[ P(x^{k}) +T\right]x^{k+1}=b,   \qquad k=0, 1,  \ldots , 
\end{equation}
which  define formally the {\it semi-smooth Newton  sequence} $\{x^{k}\}$    for solving  \eqref{eq:pwls}. Note that the above
iteration is exactly the one stated in equation (6) of \cite{BrugnanoCasulli2007}.  We devote the rest of this section to studying  the convergence properties of this sequence.
\begin{proposition} \label{pr:bounded}
        Assume that the matrix $P(x) +T$    is nonsingular for all $x\in \mathbb{R}^n$.  Then,   $\{x^{k}\}$  is well defined and bounded from any starting point. Moreover, for each accumulation point $\bar x$ of  $\{x^{k}\}$ there exists 
	 an $\hat x \in \mathbb{R}^n$ such that
\begin{equation} \label{eq:sap2}
\left[ P(\hat{x}) +T \right]{\bar x}=b.
\end{equation}
	 In particular, if   $\mbox{sgn}(\bar x^+)=\mbox{sgn}(\hat x^+)$,   then $ \bar x$  is a solution of \eqref{eq:pwls}.

\end{proposition}
\begin{proof}
To  prove  this result we follow similar arguments of Proposition 3 of \cite{Mangasarian2009}.
\end{proof}
\noindent The next proposition gives a condition for the Newton iteration \eqref{eq:newtonc2pw} to finish in a finite
number of steps, which can be proved by using the same argument as the one used in the proof of Lemma 3 of \cite{BrugnanoCasulli2007}.
\begin{proposition} \label{pr:ft}
	If in \eqref{eq:newtonc2pw} it happens that $\mbox{sgn}((x^{k+1})^+)=\mbox{sgn}((x^{k})^+)$,   then $x^{k+1}$ is a solution of \eqref{eq:pwls}.
\end{proposition}
Next,  we state  and  prove a theorem for  the semi-smooth Newton's method   \eqref{eq:newtonc2pw} for solving \eqref{eq:pwls}. 
\begin{theorem}\label{th:mrqpw}
Let  $b \in \R^n$ and  $T \in \R^{n\times n}$ be a nonsingular matrix.   Assume that $\left\|T^{-1}\right\|<1$. Then,   for any starting point $x^0 \in \R^{n}$,   $\{x^{k}\}$    is well-defined. Additionally,  if
\begin{equation} \label{eq:CC2pw}
\left\|T^{-1}\right\|<1/2, 
\end{equation}
then  $\{x^{k}\}$    converges $Q$-linearly  to  $x^*\in \mathbb{R}^n$,  the unique solution  of  \eqref{eq:pwls}, as follows 
  \begin{equation} \label{eq:lconv3pw}
\|x^*-x^{k+1}\|\leq \frac{\|T^{-1}\|}{1- \|T^{-1}\|} \|x^*-x^{k}\|,  \qquad k=0, 1, \ldots .
\end{equation}
\end{theorem}
\begin{proof} 
Let  $x \in \mathbb{R}^n$. Since  $\left\|T^{-1}\right\|<1$, the definition of $P(x)$ implies $\|T^{-1}P(x)\|\leq \|T^{-1}\|<1 $.  Thus,    Lemma~\ref{lem:ban}  implies that  $-T^{-1}P(x) -{\rm I}$ is nonsingular.  Because $T$  is   nonsingular and 
 \[
  P(x)+T=-T\left[ -T^{-1}P(x)-{\rm I}\right], \qquad  \quad ~x \in \mathbb{R}^n, 
 \]
  we conclude  that  $ P(x)+T$ is also nonsingular. Hence, for any starting point $x^0 \in \R^{n}$, \eqref{eq:newtonc2pw} implies that  $\{x^{k}\}$    is well-defined  .

 Using  Proposition~\ref{pr:uniqqpw},  we conclude that \eqref{eq:pwls} has a unique solution $x^*\in  \R^{n}$.  Since $x^*\in  \R^{n}$ is the solution of \eqref{eq:pwls},   we have $ [P(x^*)+T]x^*-b=0$, which together  with definition of   $\{x^{k}\}$ in \eqref{eq:newtonc2pw} and \eqref{eq:mppw} implies
\[
x^*-x^{k+1}= -[P(x^{k})+T]^{-1}\big[ [P(x^*)+T]x^*-b - [P(x^{k})+T]x^{k}+b- [P(x^{k})+T](x^*-x^{k})\big],  \quad k=0, 1, \ldots .
\]
On the other hand, since $ P(x)x=x^+ $ for all $x\in \R^n$, after  simple algebraic manipulations we obtain 
\[
[P(x^*)+T]x^*-b - [P(x^{k})+T]x^{k}+b- [P(x^{k})+T](x^*-x^{k})= (x^*)^+ - (x^{k})^+-P(x^{k})(x^*-x^{k}),
\]
for $k=0, 1, \ldots $. Combining the two above equalities and   using properties of the norm   we have
\[
\|x^*-x^{k+1}\|\leq \left \|P(x^{k})+T]^{-1}\right\|  \left\| (x^*)^+ - (x^{k})^+-P(x^{k})(x^*-x^{k})\right\|,  \qquad k=0, 1, \ldots .
\]
It follows from Lemma~\ref{l:error} that  $ \left\| (x^*)^+ - (x^{k})^+-P(x^{k})(x^*-x^{k})\right\|\leq \|x^*-x^{k}\|$, for $ k=0, 1, \ldots$, and the last inequality   becomes 
\begin{equation} \label{eq:conv}
\|x^*-x^{k+1}\|\leq \left \|[P(x^{k})+T]^{-1}\right\|  \|x^*-x^{k}\|, \qquad k=0, 1, \ldots .
\end{equation}
On the other hand, after some algebra and using properties of the norm,  we have
\[
\left\| [P(x^{k})+T]^{-1}\right\|=  \left\| [-T^{-1}P(x^{k})-{\rm I}]^{-1} \left[-T^{-1}\right]\right\|\leq  \left\| [T^{-1}P(x^{k})+{\rm I}]^{-1}\right\| \|T^{-1}\|,\quad k=0, 1, \ldots , 
\]
which combined with   Lemma~\ref{lem:ban} and considering that $\|T^{-1}P(x^{k})\|\leq \|T^{-1}\|<1 $, implies  
\[
\left\| [P(x^{k})+T]^{-1}\right\| \leq \frac{ \|T^{-1}\|}{1- \left\|T^{-1}\right\|}, \qquad \quad k=0, 1, \ldots .
\]
Thus, last inequality together with  \eqref{eq:conv} gives   \eqref{eq:lconv3pw}. Assumption \eqref{eq:CC2pw} implies  $\|T^{-1}\|/(1- \|T^{-1}\|)<1$. Therefore,   \eqref{eq:lconv3pw} implies that $\{x^{k}\}$ converges Q-linearly,  from any starting point $x^0$,  to the 
solution $x^*$ of  \eqref{eq:pwls}. Hence  the theorem is proven. 
\end{proof}
 For stating the next result we need the following definition. Let $S:=\left(s_{ij}\right)\in \R^{n \times n}$  be with $i-$th row
$s_i:=(s_{i1}, \ldots, s_{in})^T$, $i=1,\dots,n$. We say that $S$ has {\it }, if for each $i-$th row $s_i$ either $s_i\geq 0$
or $s_i\leq 0$.
\begin{example}
The following three matrices have  its rows with definite sign:
$$
\begin{bmatrix}
-2 & -3  &  -1 \\
1 & 1 & 2 \\
5 & 2 & 1
\end{bmatrix},
\qquad
\begin{bmatrix}
2 & 3  &  1 \\
1 & 1 & 2 \\
5 & 2 & 1
\end{bmatrix},
\qquad 
\begin{bmatrix}
-2 & -3  &  -1 \\
-1 & -1 & -2 \\
-5 & -2 & -1
\end{bmatrix}.
$$
\end{example}
\begin{example}
It follows from \cite[Theorem 2]{Plemmons1977}  that,   if $A \in \R^{n \times n}$ is a non-singular  $M$-matrix then  $(A + D)^{-1}\geq 0$, for
each diagonal matrix  $D \in \R^{n \times n}$ with $D\geq 0$. In particular,  if $A \in \R^{n \times n}$ is an $M$-matrix, then  $(A + D)^{-1}$ has  its rows with definite sign,   for each $D\geq 0$.
\end{example}

\begin{theorem}\label{teo-finite}
Assume that \eqref{eq:pwls} has solutions. If $[P(x)+T]^{-1}$ exists  and  have its rows with definite sign,   for all $x\in\R^n$. Then  $\{x^{k}\}$  generated by  \eqref{eq:newtonc2pw} converges after finite steps for the unique solution of \eqref{eq:pwls}.
\end{theorem}
\begin{proof} First of all note that the sequence generated by \eqref{eq:newtonc2pw} satisfies
	\begin{equation}\label{consequencia-newton}
    F(x^{k})+[P(x^{k})+T](x^{k+1}-x^{k})=0, \qquad  k=0,1,\ldots, 
	\end{equation}
where the function $F$ is defined in \eqref{eq:fucpw}. By direct computation, we have
\begin{equation}\label{eq-chave}
F(y)-F(x)- \left[P(x)+T\right](y-x)= P(y)y-P(x)y\ge0, \qquad \qquad ~ x, y \in \mathbb{R}^n.
\end{equation} 
For arbitrary $x^{0}\in\R^n$, the above inequality and \eqref{consequencia-newton} imply that
$$
F(x^1)\ge F(x^{0})+[P(x^{0})+T](x^1-x^{0})=0.
$$
Thus, applying an induction argument we conclude that
\begin{equation}\label{Fnegativa}
F(x^{k})=[P(x^{k})+T]x^{k}-b\ge0, \qquad \qquad \, k=1,2,\ldots.
\end{equation}
Let $x^{*}$ be a solution of \eqref{eq:pwls}.   Letting $y=x^{*}$ and $x=x^{k}$ in    \eqref{eq-chave}, we obtain
\begin{equation}\label{13insolution}
0= F(x^{*})\ge F(x^{k})+[P(x^{k})+T](x^{*}-x^{k}).
\end{equation} 
Since $s_i=(s_{i1}, \ldots, s_{in})^T$,  the $i-$th row of $[P(x)+T]^{-1}=:(s_{ij})$,  has all elements  either non-negative or non-positive, we have only two options: $\sgn(s_i^T)$ has its components equal to $-1$ or $0$, or $\sgn(s_i^T)$ has its components equal $0$ or $1$. Multiplying  both sides of \eqref{13insolution} by $[P(x^{k})+T]^{-1}$ and using \eqref{Fnegativa}, we have
\begin{equation}\label{I+}
x_i^*\le x_i^k-s_iF(x^{k})\le x_i^k, \qquad \qquad i\in I_+:=\left\{1\le i\le n : \sgn(s_i^T)\in\{0,1\}\right\}, 
\end{equation} 
for all $k\ge1$, and similarly   
\begin{equation}\label{I-}
x_i^*\ge x_i^k-s_iF(x^{k})\ge x_i^k, \qquad \qquad \;i\in I_-:=\left\{1\le i\le n : \sgn(s_i^T)\in\{-1,0\}\right\}.
\end{equation} 
Note that as $[T+P(x^{k})]^{-1}$ exists, then there are no indexes $i$ and $j$ such that $s_i=s_j$, thus $I_+\cap I_-=\emptyset$ and $I_+\cup I_-=\{1,2\ldots,n\}$. It follows from \eqref{eq:nmqcpw},  \eqref{eq:newtonc2pw} and $V^k=P(x^{k}) +T$ that
$$
x^{k+1}=[T+P(x^{k})]^{-1}b=x^{k}-[P(x^{k})+T]^{-1}F(x^{k}),\qquad \qquad \, k=0,1,\ldots.
$$
Therefore,  using \eqref{Fnegativa} and  the definition of  $I_+$, we obtain
\begin{equation}\label{monotoniaI+}
x_i^*\le x_i^{k+1}\le x_i^k, \qquad \qquad \; i\in I_+,
\end{equation} where the first inequality above follows from  \eqref{I+}, and analogously  using \eqref{I-}, we have
\begin{equation}\label{monotoniaI-}
x_i^*\ge x_i^{k+1}\ge x_i^k, \qquad \quad  \; i\in I_-.
\end{equation}
Hence, $\{x^{k}\}$ converges, because $\{x_i^k\}$   is monotone and bounded by $x_i^*$  for $i=1,\ldots,n$. Thus, $\{x_i^k\}$ has a limit $u_i$. Therefore, $\{x^{k}\}$ converges to the vector  $u$ with components  $u_i$. By using again \eqref{consequencia-newton}, we have
\begin{equation*}
\|F(u)\|=\lim_{k\to\infty}\|F(x^{k})\|=\lim_{k\to\infty}\|[P(x^{k})+T](x^{k+1}-x^{k})\|\le(1+\|T\|)\lim_{k\to\infty}\|x^{k+1}-x^{k}\|=0.
\end{equation*} 
Therefore, $u$ is a solution. Furthermore, for any two solutions $x^{*}$ and $y^*$, \eqref{eq-chave} implies 
\begin{equation}\label{sol-unica}
0=F(x^{*})-F(y^*)\ge[T+P(y^*)](x^{*}-y^*).
\end{equation} Then, multiplying by $[T+P(y^*)]^{-1}$ we obtain
$$
y_i^*\ge x_i^*\quad \quad\; i\in I_+ \qquad \mbox{and} \qquad y_i^*\le x_i^*\quad \quad\; i\in I_-.
$$
The result follows by reversing the roles of $x^{*}$ and $y^*$ in \eqref{sol-unica}. Thus, the problem has a unique solution equal to the limit of the sequence $\{x^{k}\}$ generated by \eqref{eq:newtonc2pw}.

Finally we establish the finite  termination of the sequence $\{x^{k}\}$ at the unique solution of problem \eqref{eq:pwls}, which will be denoted by $x^{*}$. Since for all $x \in \mathbb{R}^n$ $P(x)$ has at most $2^n$ different choices,    then there
exist $j, \ell \in\NN$ with  $1\leq \ell < 2^n$ such that $P(x^j)=P(x^{j+\ell})$. Note that if $\ell=1$, then
Proposition~\ref{pr:ft} implies that $x^{j+2}$ is solution of \eqref{eq:pwls}.     This statement implies that 
$$
x^{j+1}=[T+P(x^j)]^{-1}b=[T+P(x^{j+\ell})]^{-1}b=x^{j+\ell+1}.
$$
Applying inductively this argument, 
$$
x^{j+1}=x^{j+\ell+1}, \;\; x^{j+2}=x^{j+\ell+2}, \;\;\ldots, \;x^{j+\ell}=x^{j+2\ell}, \;\;x^{j+\ell+1}=x^{j+2\ell+1}=x^{j+1}.
$$ Thus, the sequence $\{x^k\}$ generated by \eqref{eq:newtonc2pw} has at most $j+\ell$ different elements. Now using \eqref{monotoniaI+} and \eqref{monotoniaI-}, we obtain
$$
x_i^{j+1}\ge x_i^{j+2}\ge \cdots \ge x_i^{j+\ell+1}=x_i^{j+1}, \qquad \qquad \;i\in I_+,
$$ and $$
x_i^{j+1}\le x_i^{j+2}\le \cdots \le x_i^{j+\ell+1}=x_i^{j+1}, \qquad \quad \;i\in I_-.
$$ 
Hence, $x^{j+1}= x^{j+2}$ and in view Proposition~\ref{pr:ft} we conclude that  $x^{j+2}$ is solution of \eqref{eq:pwls}, i.e., $x^{j+2}=x^{*}$. 
\end{proof}
It is worth mentioning that Theorem \ref{teo-finite} generalizes Theorem 2 of \cite{BrugnanoCasulli2007}, in the special case  $[P(x)+T]^{-1}\geq 0$, for all $x\in \mathbb{R}^n$. The invertibility  of $P(x)+T$,  for all $x\in \mathbb{R}^n$,  is sufficient to the  well-definedness   of the  semi-smooth Newton method. However,   the next  example  show that, for the convergence of these methods,  an additional condition on  $T$ must be assumed, for instance,   \eqref{eq:CC2pw} or  $[P(x)+T]^{-1}$ exists with its rows having  definite sign,  for all $x\in\R^n$. 
\begin{example}
Consider the  function $F: \mathbb{R}^2 \to \mathbb{R}^2$ defined by  $F(x)=x^+ + Tx- b$, where
$$
T=  \begin{bmatrix}
-2 & 3 \\
-1 & 1
\end{bmatrix},  \qquad 
b= \begin{bmatrix}
-5\\
-3
\end{bmatrix}.
$$
Note that  $\|T^{-1}\|=3,86...$, the matrix $P(x)+T$ is invertible and have no rows with  definite sign,  for all $x\in \mathbb{R}^2$. Moreover,   $F$ has $x^{*}=[2 , -1]^T$ as the unique zero.  Applying semi-smooth Newton method starting with $x^{0}=[-3, 3]^T$,  for finding the zero of $F$, the generated sequence oscillates between the points
$$
x^1=\displaystyle \begin{bmatrix}
4\\
1
\end{bmatrix}, \qquad 
x^2=\displaystyle \begin{bmatrix}
-1\\
-2
\end{bmatrix}.
$$
\end{example}
\section{Application to  quadratic   programming} \label{sec: defbp}
In this section,  we apply the results of Section~\ref{sec: defscls} to solve \eqref{equation},  in order to find a solution of \eqref{eq:napsc2}.  We begin showing  that,    from each solution of \eqref{equation}   we obtain  a solution of  \eqref{eq:napsc2}. From now on we assume that $Q$ is a symmetric and  positive definite matrix.
\begin{proposition}  \label{pr:polscq}
If the vector $x^*$  is a solution of \eqref{equation}, then  $ (x^*)^+$   is a solution of \eqref{eq:napsc2}.
\end{proposition}
\begin{proof}
The optimality conditions of  the problem in \eqref{eq:napsc2} are given  by
\begin{equation} \label{eq:ocqp}
x \in \R^n_+,  \qquad  Qx + {\tilde b} \in  \mathbb{R}^n_+, \qquad      \left\langle  Qx + {\tilde b} , x \right\rangle =0.
\end{equation}
We claim that $ (x^*)^+$   is a solution of \eqref{eq:ocqp}. We know  that  $(x^*)^+-x^*=(x^*)^-$. Thus,  if $x^*\in \R^n$ is a solution of \eqref{equation}, then
\[
Q(x^*)^+ + {\tilde b}=(x^*)^-.
\]
Hence,  by using  $(x^*)^- \in \mathbb{R}^n_+$  and  $\left\langle   (x^*)^-,  (x^*)^+ \right\rangle=0$, the last equality easily  implies that 
\[
Q(x^*)^+ + {\tilde b}  \in \mathbb{R}^n_+, \qquad \langle  Q(x^*)^+ + {\tilde b},  (x^*)^+ \rangle =  0.
\]
Combining this with  $(x^*)^+ \in\mathbb{R}^n_+$,   we conclude that  $ (x^*)^+$  is a solution of \eqref{eq:ocqp}  as claimed, which completes  the proof.  
\end{proof} 
\noindent The   {\it  semi-smooth Newton method} for solving \eqref{equation}, with starting point   $x^{0}\in \mathbb{R}^n$,  is  given by 
 \begin{equation} \label{eq:newtonc2}
	 x^{k+1}=-\left ( \left[ Q -{\rm I}\right]P(x^{k}) +{\rm I}\right )^{-1} {\tilde b},  \qquad k=0,1,\ldots . 
\end{equation}
\begin{remark} \label{eq:eqv}
If  $Q-{\rm I} $ is  a  nonsingular  matrix,  $T=[Q-{\rm I}]^{-1}$ and $b=- T{\tilde b}$, then    \eqref{equation} and \eqref{eq:pwls} are equivalent.  Moreover, \eqref{eq:newtonc2} becomes 
 \[
 x^{k+1}=\left[T^{-1}P(x^{k}) +{\rm I}\right ]^{-1} T^{-1}b= \left[ P(x^{k}) +T\right ]^{-1} b,  \qquad k=0,1,\ldots , 
\]
which is the  semi-smooth Newton method defined in   \eqref{eq:newtonc2pw}.
\end{remark}
\begin{proposition}\label{pr:uniqq}  
Let $ \lambda \in \mathbb{R}$.       If $\left\| Q -{\rm I}\right\|\leq \lambda<1 $   then  \eqref{equation}  has a unique solution. 
\end{proposition}
\begin{proof} 
The proof follows by combining    Remark~\ref{eq:eqv} with Proposition~\ref{pr:uniqqpw}.
\end{proof}
\noindent The next result   shows that the semi-smooth Newton defined in \eqref{eq:newtonc2} is always well defined. 
\begin{lemma}\label{nonsingGC}
Let  $x\in \mathbb{R}^n$. The following    matrix is nonsingular 
\begin{equation} \label{eq:mnm}
\left[ Q -{\rm I}\right]P(x) +{\rm I}. 
\end{equation}
 As a consequence,  the  semi-smooth Newton sequence $\{x^{k}\}$    is well-defined,  for any starting point $x^{0} \in \R^{n}$.
\end{lemma}
\begin{proof}
The proof of    the  first  part of the lemma,  follows  similar argument to the proof of   Lemma~5 of \cite{FerreiraNemeth2014}. To prove 	the second part of the lemma, combine  the definition of  $\{x^{k}\}$ in \eqref{eq:newtonc2} and the first part of the lemma.
\end{proof}
\begin{proposition} \label{pr:ftq}
	If in \eqref{eq:newtonc2} it happens that $\mbox{sgn}((x^{k+1})^+)=\mbox{sgn}((x^{k})^+)$,   then $x^{k+1}$ is a solution of \eqref{equation}. 
\end{proposition}
\begin{proof}
The proof follows  combining   Remark~\ref{eq:eqv} and  Proposition~\ref{pr:ft}.
\end{proof}
\begin{proposition} 
         The  sequence $\{x^{k}\}$,  defined in \eqref{eq:newtonc2},   is bounded from any starting point. Moreover, for each accumulation point $\bar x$ of  $\{x^{k}\}$,  there exists 
	 an $\hat x \in \mathbb{R}^n$ such that
\begin{equation} \label{eq:sap}
 \left( \left[Q - {\rm I}\right]P(\hat{x}) +{\rm I}\right){\bar x}=-\tilde{b}.
\end{equation}
	 In particular, if   $\mbox{sgn}(\bar x^+)=\mbox{sgn}(\hat x^+)$   then $ \bar x$  is a solution of \eqref{equation}.
\end{proposition}
\begin{proof}
Using Remark~\ref{eq:eqv} and Proposition~\ref{pr:bounded} the result follows.
\end{proof}
\begin{theorem}\label{th:mrq}
The sequences $\{x^{k}\}$  generated by the semi-smooth Newton Method \eqref{eq:newtonc2} for solving   \eqref{equation}, is well defined  for  any  starting  point   $x^{0}\in \mathbb{R}^n$.  Moreover, if 
\begin{equation} \label{eq:CC2}
\left\|Q -{\rm I}\right\|<1/2, 
\end{equation}
 then  the sequence $ \{x^{k}\}$   converges $Q$-linearly  to  $x^*\in \mathbb{R}^n$,  the unique solution  of \eqref{equation}, as follows 
  \begin{equation} \label{eq:lconv3}
\|x^*-x^{k+1}\|\leq \frac{\| Q -{\rm I}\|}{1- \|Q -{\rm I}\|} \|x^*-x^{k}\|,  \qquad k=0, 1, \ldots ,
\end{equation}
and $ (x^*)^+$   is a solution of \eqref{eq:napsc2}.
\end{theorem}
\begin{proof}
The well-definedness,  for  any  starting  point   $x^{0}\in \mathbb{R}^n$, follows from Lemma~\ref{nonsingGC}. For concluding the proof  combine, Proposition~\ref{pr:polscq},   Remark~\ref{eq:eqv} and  Theorem~\ref{th:mrqpw}.
\end{proof}
Note that \eqref{eq:CC2} implies that the eigenvalues of $Q$ belong to  $(0, \frac{1}{2})\cup (\frac{1}{2},\frac{3}{2})$.  Let  us  present  an important  equivalent form  of problem  \eqref{eq:napsc2}. 
\begin{example}
 Given  $A \in \R^{n\times n}$    a nonsingular matrix,  $A\R^n_+:=\{Ax~:~ x\in\R^n_+\}$  and  $z\in \R^n$.   The {\it projection $P_{A\R^n_+}(z)$ of the point $z$ onto the cone  $ A\R^n_+ $} is defined by 
\[
P_{A\R^n_+}(z):=\argmin \left\{ \frac{1}{2}\|z-y\|^2~:~ y\in A\R^n_+\right\}.
\]
From the definition of the simplical cone associated with  the  matrix   $A$, the problem of   projecting the point $z\in \R^n$
onto a  simplicial cone   $A\R^n_+$ may be  stated  as the following  positively constrained  quadratic programming problem 
\begin{align*}
 & \min  ~ \frac{1}{2}\|z-Ax\|^2, \\ 
 & \const  x \in \R^n_+.
\end{align*}
Hence, if $v\in \R^n$ is the   unique  solution of this problem then  we have   $P_{A\R^n_+}(z)=Av$.  The above problem is equivalent to the
following   nonegatively constrained  quadratic programming problem
\begin{align}  \label{eq:nap}
 & \min  ~ \frac{1}{2}x^\top Qx+x^\top \tilde{b} +c \\  \notag
 & \const  x \in \R^n_+, 
\end{align}
by taking $Q=A^\top A$,  $\tilde{b}=-A^\top z$ and $c=z^\top z /2$.  The  optimality condition for problem \eqref{eq:nap}  implies  that its solution  can be obtained by solving the following linear complementarity problem 
\begin{equation} \label{eq:lcp}
y - Qx= \tilde{b}, \qquad x \ge 0, \qquad    y \ge0, \qquad    \langle x, y \rangle=0.
\end{equation} 
\end{example}
\begin{remark}
It is  easy to establish  that corresponding to each  nonnegative quadratic problems  \eqref{eq:nap} and each  linear complementarity problems \eqref{eq:lcp} associated to positive definite matrices, there are equivalent  problems of   projection onto  simplicial cones.   Therefore, the problem of projecting onto  simplicial cones can be solved by  active set
methods \cite{Bazaraa2006,LiuFathi2011,LiuFathi2012,Murty1988} or any algorithms for solving LCPs,  see \emph{e.g.},  \cite{Bazaraa2006,Murty1988} and
special methods based on its geometry,   see \emph{e.g.}, \cite{MurtyFathi1982,Murty1988}. Other fashionable ways to solve this problem are based on
the classical von Neumann algorithm (see \emph{e.g.}, the Dykstra algorithm \cite{DeutschHundal1994,Dykstra83,Shusheng2000}). Nevertheless, these 
methods are also quite expensive (see the numerical results in \cite{Morillas2005} and the remark preceding section 6.3 in  \cite{MingGuo-LiangHong-BinKaiWang2007})
\end{remark}
\section{Computational results} \label{sec:ctest}
\label{sec:computationalresults}

In this section we test our semi-smooth Newton method (\ref{eq:newtonc2}) to find solutions on generated random instances of \eqref{equation}. We present two types of experiments. In one of them, we guarantee that for each test problem the hypotheses given in 
Theorem~\ref{th:mrq} are satisfied and in the other they are not.

 All programs were implemented in MATLAB Version 7.11 64-bit and run on a $3.40 GHz$ Intel Core $i5-4670$ with $8.0GB$ of RAM. All MATLAB codes and generated data of this paper are available in \url{http://orizon.mat.ufg.br/pages/34449-publications}. 

All experiments are based on the following general considerations:
\begin{itemize}
    \item In order to accurately measure the method's runtime for a problem, each one of the test problems  was solved $10$ times and the runtime data collected. Then, we defined the corresponding  {\it method's  runtime for a problem} as the median of these measurements.

\item Let $\Tol X\in \R_+$ be a relative bound, we consider that the method converged to the solution and stopped the execution when, for some $k$, the condition 
$$
\|u-x^{k}\|<\Tol X(1+\|u\|), 
$$ is satisfied. If the previous stopping criteria are not met within $100$ iterations, we declare that the method did not converge.
\end{itemize}

\subsection{When the hypotheses of Theorem~\ref{th:mrq} are satisfied}
In this experiment, we studied the behavior of the method on sets of $100$ randomly generated test problems of dimension $n=2000,3000,4000,5000$,
respectively. Furthermore, we analyzed the influence of the initial point in the convergence of the method on $1000$ randomly generated test problems of dimension
$n = 100$. 
For each test problem in this experiment the hypotheses given in the Theorem~\ref{th:mrq} are satisfied, generating each of them as follows:
\begin{enumerate}
  \item    To construct the matrix $Q\in \R^{n\times n}$ symmetric and positive definite  satisfying  the assumption (\ref{eq:CC2}) in Theorem~\ref{th:mrq}, we first chose a
	     random number $\beta$ from the standard uniform distribution on the open interval $(0,1/2)$.  Secondly,  we compute  the 
	     singular value decomposition $U \Sigma V^{T}$ of a symmetric and positive definite matrix of the form $B^{T}B$, where $B$ is a generated $n\times n$ real nonsingular matrix containing random values drawn 
	     from the uniform distribution on the interval $[-10^6,10^6]$. Finally,  in
the present case the equality $V= U$ holds and we compute the matrix $Q$ from 
$$
Q =U~\left(I+\frac{\beta}{\sigma} \Sigma\right)~U^T, 
$$
where  $\sigma$ is the largest singular value of $\Sigma$. It is important to note that by construction of the matrix $Q$ always $\beta=\left\|Q -{\rm I}\right\|$.

\item   We have chosen the solution $u\in \R^{n}$ containing random values drawn from the uniform distribution on the interval $[-10^6,10^6]$ and
	then we have computed $\tilde{b}\in \R^{n}$ from equation (\ref{equation}). 
\item Finally we have chosen a  starting point $x^{0}\in \R^{n}$ containing random values drawn from the uniform distribution on the interval $[-10^6,10^6]$.
\end{enumerate}

In accordance with the theoretical convergence of the method, ensured by Theorem~\ref{th:mrq}, the computational convergence is obtained in all cases.

The computational results to analyze the behavior of the method on sets of $100$ generated random test problems of different dimensions, are reported in Table~\ref{tab:example1}. From these, it can be noted that for the same dimension, to achieve higher accuracy, the method does not experience a significant increase in the number of iterations or in runtime. On the other hand, the increase in the dimension of the problems does not necessarily involve  an increase in the number of iterations to achieve the same accuracy, however, a larger runtime is consumed. A larger runtime consumption is associated with the fact that the semi-smooth Newton method~(\ref{eq:newtonc2}) requires the solution of a linear system in each iteration, whose computational effort increases with the dimension of the problem. Another important aspect that can be checked in Table~\ref{tab:example1} is the ability of the method to converge  in about three iterations on average.

\begin{table}[htbp]
\centering
\renewcommand{\arraystretch}{1.2}
\begin{tabular}{|c|lll|rrr|}
\hline
\multicolumn{1}{|c|}{$n$}&\multicolumn{3}{c|}{Total Iterations}&\multicolumn{3}{c|}{Total Time}\\\cline{1-7}
 $2000$& 278 & 294 & 296 &  142.48&  147.77 &  148.05 \\ 
 $3000$& 282 & 295& 299 &  445.61 &  465.48 &  471.65 \\
 $4000$& 278 & 297 & 300 & 1013.79 & 1082.43 & 1093.55 \\
 $5000$& 285 & 303 & 307 & 1945.23 & 2067.08 & 2112.72 \\
\hline\hline
\multicolumn{1}{|c|}{$\Tol X$}&\multicolumn{1}{l}{$10^{-6}$}&\multicolumn{1}{l}{$10^{-8}$}&\multicolumn{1}{l|}{$10^{-10}$}&\multicolumn{1}{c}{$10^{-6}$}&\multicolumn{1}{c}{$10^{-8}$}&\multicolumn{1}{c|}{$10^{-10}$}\\
\hline
\end{tabular}
\caption{ Total overall iterations and total time in seconds, performed and consumed, respectively by the semi-smooth Newton method (\ref{eq:newtonc2})  to solve the $100$ test problems  of each dimension for different accuracies.}
\label{tab:example1}
\end{table} 

In order to study the influence of the initial point in the convergence of the method,  we have generated $1000$ test problems of dimension 
$n = 100$ and we have associated to each of them $1000$ generated initial points. We have solved each problem with the $1000$ 
corresponding initial points. Then, we have computed the standard deviation ($\mbox{STD}$) $\overline{d}_i$ and the mean value (\mbox{MEAN})
$\overline{m}_i$ of the number of iterations performed by the method to solve the problem $i$ taking each one of the $1000$ initial points.
Finally we have computed the mean of all $\overline{d}_i$  and the mean of all $\overline{m}_i,\;\;i=1,...,1000$. All cases converged, indicating
robustness of the method with respect to the starting point. The results are shown in Table~\ref{tab:example2}. The standard deviation of  the
number of iterations performed by the method to solve the problem $i$ with  the $1000$ initial points  gives us an idea of the influence of the
initial point in the number of iterations performed by the method in each problem. The reported means of these  standard deviation values give 
us an idea of the influence of the initial point in the number  of iterations performed by the method in all the problems in general. The results
in the table show that on average the number of iterations performed by our method to find the solution for a problem varies only very slightly 
with the chosen starting point. Again we see that the average number of iterations performed is less than three.

\begin{table}[htbp]
\centering
\renewcommand{\arraystretch}{1.2}
\begin{tabular}{|c|c|c|}
\hline
Tol X&\multicolumn{1}{c|}{$\mbox{MEAN}\left(\{\overline{d}_i\}_{i=1,...,1000}\right)$}&\multicolumn{1}{c|}{$\mbox{MEAN}\left(\{\overline{m}_i\}_{i=1,...,1000}\right)$}\\\hline\hline
$10^{-6}$&0.2450 & 2.3331\\
$10^{-8}$&0.2530 &2.3454\\
$10^{-10}$&0.2536 & 2.3457\\\hline
\end{tabular}
\caption{ Influence of the initial point in the convergence of the semi-smooth Newton method (\ref{eq:newtonc2}) on a total of  $1000$ test problems of dimension $n=100$ each of them with $1000$ generated initial points for different accuracies.}
\label{tab:example2}
\end{table} 

\subsection{When the hypotheses of Theorem~\ref{th:mrq} are not satisfied} 
In this experiment, we studied the behavior of the method on $1000$ test problems of dimension $n=1000$, where the hypotheses given in the Theorem~\ref{th:mrq} are not all satisfied.

In this case, the test problems were built almost as in the previous experiment. The only difference was in the construction of the matrix
$Q\in \R^{n\times n}$ not satisfying the assumption (\ref{eq:CC2}) of Theorem~\ref{th:mrq}. Namely, we chose the random number $\beta$ from the
standard uniform distribution on the  interval $[lb,ub)$, where $\frac{1}{2}\leq lb<ub$.

According to the obtained numerical results, we can conjecture that our method converges to a much broader class of problems, not satisfying the
hypotheses of Theorem~\ref{th:mrq}. However we detected that convergence with high accuracy to the solution largely depends on the magnitude of
the value of the norm in condition (\ref{eq:CC2}). This idea can be observed inspecting Table \ref{tab:example3}. As the magnitude of the value
of the norm in (\ref{eq:CC2}) increases sufficiently, the number of
problems for which the method converges to the solution with greater accuracy decreases. This phenomenon, of course, is not associated to the convergence of the method for a specific problem, but, rather, there is an optimum accuracy achievable due to the accumulated errors. Small tolerances do not ensure obtaining accurate results. It can be the case that convergence is overlooked and unnecessary iterations are performed. It is important to note in the table that, even when the hypothesis is unfulfilled, the method converges for these problems, however it can be noted that the number of iterations performed by the method increases with respect of the previous experiments in which the hypotheses were fulfilled.

\begin{table}[htbp]
\centering
\renewcommand{\arraystretch}{1.2}
\begin{tabular}{|c|rrr|rrr|}
\hline
\multicolumn{1}{|c|}{$\beta\in[lb,ub)$}&\multicolumn{3}{c|}{Solved Problems}&\multicolumn{3}{c|}{Iterations}\\\cline{1-7}
 $[0.5,10^3)$ &1000 & 1000 & 1000 & 7.2160 & 7.2190 & 7.2190\\
 $[10^3,10^4)$&1000 & 1000 & 1000 & 9.1800 & 9.1850 & 9.1850\\
 $[10^4,10^5)$& 1000 & 1000 & 1000 & 9.6730 & 9.6760 & 9.6760\\
 $[10^5,10^6)$&1000 & 1000 & 693  & 10.2820 & 10.2860 & 10.2540\\
 $[10^6,10^7)$& 1000 & 999 & 0    & 10.3870 & 10.3874 &   -\\
 $[10^7,10^8)$& 998 & 690  & 0    & 10.4339 & 10.4246 &  -\\\hline\hline
\multicolumn{1}{|c|}{Tol X}&\multicolumn{1}{c|}{$10^{-6}$}&\multicolumn{1}{c|}{$10^{-8}$}&\multicolumn{1}{c|}{$10^{-10}$}&\multicolumn{1}{c|}{$10^{-6}$}&\multicolumn{1}{c|}{$10^{-8}$}&\multicolumn{1}{c|}{$10^{-10}$}\\
\hline
\end{tabular}
\caption{Number of problems solved by the semi-smooth Newton method (\ref{eq:newtonc2}) on a total of  $1000$ test problems of dimension $n=1000$ of each condition ($lb\leq \beta <ub$) for different accuracies, and the mean number of iterations performed by the semi-smooth Newton method (\ref{eq:newtonc2})  to solve one problem in each case.}
\label{tab:example3}
\end{table} 
\section{Conclusions} \label{sec:conclusions}
In this paper we studied  a special class of  convex quadratic programming under positive constraint, which,  via  its optimality conditions, is reduced to  finding the unique solution of a  nonsmooth system of equations.  Our main 
result shows that, under a mild assumption on the simplicial cone, we can  apply  a semi-smooth Newton method for finding a unique solution of  
the  obtained  associated nonsmooth system of equations and  that  the generated sequence converges linearly to the  solution   for  any starting
point.  It would be interesting to see whether the used technique can be applied for solving  more general convex programs.    

Since the optimality condition of a positive  constrained convex quadratic programming problem is equivalent to a   linear
complementarity problem, which is equivalent to the problem of finding the unique solution of a nonsmooth system of equations, another   interesting problem to address is to compare our semi-smooth Newton method   with   active set methods  
\cite{Bazaraa2006,LiuFathi2011,LiuFathi2012,Murty1988}. 

This paper is a continuation of  \cite{FerreiraNemeth2014}, where we studied the problem of projection onto a simplicial cone by using 
a semi-smooth Newton method. We expect that the results of this paper become a further step towards solving general convex optimization problems. We foresee further progress in this topic in the near future.

\end{document}